\documentclass[a4paper]{article}

\title{Estimates of eigenvalues of Schr\"{o}dinger operators on the half-line with complex-valued potentials}
\author{
Alexandra Enblom\\
{\small Department of Mathematics}\\
{\small Linköping University}\\
{\small SE-581 83 Link\"oping, Sweden}\\
{\small \texttt{alexandra.enblom@liu.se}}\\
}

\usepackage[utf8]{inputenc}
\usepackage[T1]{fontenc}
\usepackage{amsmath}
\usepackage{amsthm}
\usepackage{amssymb} % for mathbb and more
\usepackage{mathrsfs} % for mathscr
\usepackage{enumerate}

\theoremstyle{plain}
\newtheorem{thm}{Theorem}[section]

\newtheorem{cor}[thm]{Corollary}

\theoremstyle{definition}

\newtheorem{rem}[thm]{Remark}

\newcommand{\reals}{\ensuremath{\mathbb{R}}}
\newcommand{\complex}{\ensuremath{\mathbb{C}}}

\DeclareMathOperator{\impart}{Im}

\newcommand{\ud}{\,d}

\numberwithin{equation}{section}

\date{}
\begin{document}
\maketitle

\begin{abstract}
Estimates for eigenvalues of  Schr\"{o}dinger operators on the half-line with  complex-valued potentials are established.  Schr\"{o}dinger operators with potentials  belonging to weak Lebesque's classes  are also considered. The  results  cover those known previously  due to R. L. Frank, A. Laptev and R. Seiringer [In spectral theory and analysis, vol. 214, Oper. Theory Adv. Appl., pag. 39-44; Birkh\"{a}user/Springer Basel.]
\end{abstract}



%%%%%%%%%%%%%%%%%%%%%%%%%%%%%%%%%%%%%%%%%%%%%%%%%%%%%%%%
\section{Introduction}\label{sec:introduction}
%%%%%%%%%%%%%%%%%%%%%%%%%%%%%%%%%%%%%%%%%%%%%%%%%%%%%%%%
 This paper is motivated by the recent  work \cite{frank-laptev-seiringer} in which estimates for  non-positive eigenvalues of Schr\"{o}dinger operators on the half-line are given.  The estimates obtained in \cite{frank-laptev-seiringer} revised  a well-known result, it is mentioned in  \cite{keller}, according to which any negative eigenvalue $\lambda$ of the (self-adjoint) Schr\"{o}dinger operator $H  (= - \ud^{2}/ \ud x^{2} + q (x))$ satisfies 
\begin{equation}  \label{eq:lambda1}
| \lambda |^{1/2}  \leq \frac{1}{2} \int_{- \infty}^{\infty} | q (x) | \ud x.
\end{equation}
This result, as was pointed out in \cite{abramov-aslanyan-davies}, remains valid for the case of  non-self-adjoint Schr\"{o}dinger operators  as well. In \cite{abramov-aslanyan-davies} it is proved \eqref{eq:lambda1}  provided that the potential $q$, being  in general a complex-valued function, belongs to $L_{1} (\reals) \cap  L_{2} (\reals).$  The mentioned  paper \cite{frank-laptev-seiringer} concerns Schr\"{o}dinger operator on $ L_{2} (\reals_{+})$ assigned with  Dirichlet (or also Neumann) boundary  conditions. Assuming that $q$ is a  summable (in general, complex-valued)  function, in \cite{frank-laptev-seiringer} instead of \eqref{eq:lambda1}  it is proved that for any (non-positive)  eigenvalue  $\lambda = | \lambda | e^{i \theta}$ $(0<\theta < 2\pi)$ of the Schr\"{o}dinger operator $H$  in $ L_{2} (\reals_{+})$ (for instance, with Dirichlet boundary condition) satisfies 
\begin{equation} \label{eq:lambda2}
| \lambda |^{1/2}  \leq \frac{1}{2}  g (\cot (\theta/2))  \int_{0}^{\infty} | q (x) | \ud x,
\end{equation}
where $g (t) : = \sup_{y \geq 0} | e^{i t y} - e^{- y} |.$

In contrast with \cite{frank-laptev-seiringer} we study the problem for the general case of  the  Schr\"{o}dinger operator $H$ considered  acting in the Banach space $L_{p} (\reals_{+})  (1 < p < \infty)$   by  assuming  that the potential $q$ admits a factorization $q = a b,$ where $a \in  L_{r} (\reals_{+})$ and $b \in L_{s} (\reals_{+})$  for some $r, s, 0 < r, s \leq \infty.$ Under additional subordinate type conditions on  the potential $q$, in order to guarantee the relatively compactness of $q$ viewed as a  perturbation of $- \ud^{2}/ \ud x^{2},$  the  Schr\"{o}dinger operator $H$  is defined as a natural  closed extension of $- \ud^{2}/ \ud x^{2} + q$  in  $ L_{p} (\reals_{+})$  having its essential spectrum (as the  unperturbed operator)  the semi-axis $\reals_{+}.$ In  this framework the problem reduces to estimation of resolvent of the unperturbed operator bordered by adequate operators of  multiplications. First we analyze  the situation of Dirichlet  boundary conditions,  and then we show that the arguments  applied are available for the general  case when mixed boundary conditions  (in particular, for Neumann boundary  conditions)  are imposed. We prove  that if $0 < r \leq \infty, p \leq s \leq \infty, r^{- 1} + s^{- 1} < 1,$ then for eigenvalues $\lambda$   lying out of  $\reals_{+}$ (the essential spectrum) there holds 
\begin{equation} \label{eq:lambda3}
| \lambda |^{1 + \alpha}  \leq    (\alpha \sin  (\theta/2))^{- 2}   \| a \|_{r}^{2 \alpha}   \| b \|_{r}^{2 \alpha},
\end{equation}
where $\alpha : = (1 - r^{- 1}  - s^{- 1})^{- 1};$  as above $\theta  = arg \lambda  \quad (0 < \theta < 2 \pi).$      In the extremal case $\alpha = \infty$ (i.e., $r^{- 1} + s^{- 1}  = 1$)  the following estimate 
\begin{equation} \label{eq:lambda4}
| \lambda |^{1/2}   \leq \frac{1}{2}  g (\cot (\theta/2))  \| a \|_{p}   \| b \|_{p^{'}}
\end{equation}
holds true, where $g$ is determined as in \eqref{eq:lambda2}; $p^{'} (= p/(p - 1))$  denotes for the conjugate exponent of  $p$. Clearly, if $a, b$ are taken as $| a | = | b | = | q |^{1/2},$  for  the  case $p = 2$, \eqref{eq:lambda4} leads to \eqref{eq:lambda2} that,   as was  already mentioned, is due to R.L.Frank, A.Laptev and R.Seiringer \cite{frank-laptev-seiringer}.

Diverse estimates useful  in applications can be  derived from the general results mentioned above.  So, letting $q =   L_{\gamma + 1/2} (\reals_{+}) $ for $\gamma > 1/2$  if $1 < p \leq 2$  and $2 \gamma > p - 1$ if $p > 2,$ the eigenvalues $\lambda = | \lambda | e^{i \theta}  (0 < \theta < 2 \pi)$  of $H$ are confined according  to the following  estimate
\begin{equation} \label{eq:lambda5}
| \lambda |^{\gamma}  \leq  \biggl( \frac{2 \gamma + 1}{2\gamma - 1}  \sin \frac{\theta}{2} \biggl)^{1/2 - \gamma}    \int_{0}^{\infty} | q (x) |^{\gamma + 1/2} \ud x.
\end{equation}
For the case of self-adjoint  Schr\"{o}dinger operators  considered on the whole line a similar  inequality to \eqref{eq:lambda5} was pointed out by E.H.Lieb and W.Thirring \cite{lieb-thirring} (cf. also the discussion undertaken in this context in \cite{frank-laptev-seiringer}; see Remark 1.6 \cite{frank-laptev-seiringer}). 

Estimates involving  decaying potentials can also be derived directly from the general results. So, if  it is taken $a (x) = (1 + x)^{- \tau}$ and  $b (x) = (1 + x)^{\tau} q (x)$ by assuming that $(1 + x)^{\tau} q \in  L_{r} (\reals_{+})$ with  $\tau r > 1,$ then the eigenvalues $\lambda$ (with $\theta = arg \lambda,  0 < \theta < 2 \pi$)  of $H$ satisfy
$$| \lambda |^{r - 1}  \leq  \frac{1}{\tau r - 1}  \biggl( \frac{r}{r - 2}  \sin \frac{\theta}{2} \biggl)^{2 - r}    \int_{0}^{\infty} |  (1 + x)^{r} q (x) |^{r} \ud x.$$
 It stands to reason that other weight functions like, for instance, $e^{\tau | x |^{\alpha}}$ with $\tau > 0, \alpha \in  \reals,$ can be also involved.

Finally, note that the arguments by  interplaying with interpolation methods \cite{bergh-lofstrom} extend the obtained results to more general  case of Schr\"{o}dinger operators  with potentials  belonging to weak Lebesgue's  spaces. A version of \eqref{eq:lambda5} for this case  is the following  one
$$| \lambda |^{\gamma}  \leq  C \sup_{t >0} (t^{\gamma + 1/2}  \lambda_{q} (t))$$
with $\gamma$ as in \eqref{eq:lambda5}; $\lambda_{q}$  denotes the distribution function of the potential $q$ with respect to the standard Lebesgue measure on $\reals_{+}.$

The paper is organized as follows. In Section \ref{sec:lebesgue} the problem is discussed  for Schr\"{o}dinger operators with Lebesgue power-summable potentials. Section \ref{sec:lebesgue} is divided in two subsections. The first is concerned with Schr\"{o}dinger operators  with the Dirichlet boundary  conditions. In the second one we discuss the general situation when the mixed boundary conditions are imposed.
In Section \ref{sec:potentials from weak} we treat the case of  potentials belonging to weak Lebesgue's  type spaces.

%%%%%%%%%%%%%%%%%%%%%%%%%%%%%%%%%%%%%%%%%%%%%%%%%%%%%%%%
\section{Lebesgue summable type potentials}  \label{sec:lebesgue}
%%%%%%%%%%%%%%%%%%%%%%%%%%%%%%%%%%%%%%%%%%%%%%%%%%%%%%%%

We consider the Schr\"{o}dinger operator $H$ defined  in the space $L_p(\reals_{+})  (1 < p < \infty)$ as a  closed extension of the formal differential operator $-\ud^{2}/\ud x^{2} + q (x).$   For it should  be  posed suitable conditions on the potential $q$ (in an averaged sense to be small at infinity) ensuring  the relatively compactness of $q$ regarded as a  perturbation operator. We assume that $q$ admits a factorization $q = a b$ with $a \in L_{r}(\reals_{+})$ and $b \in L_{s}(\reals_{+})$ for some  $0 < r, s \leq \infty.$  
 Further conditions on the potential $q$ under which the main results are obtained ensured that the essential spectrum of $H$ is the same in each of Banach space $L_p(\reals_{+})$ with $1 < p < \infty,$ and filling  the semi-axis $\reals_{+}$.  
 To this  end we restrict ourselves to refer \cite{schechter3}  for details and other  diverse related   conditions concerning general elliptic  differential operators. In this framework the problem of evaluation for eigenvalues (lying  outside of the essential spectrum) of $H$ reduces to norm estimation of the resolvent  of the unperturbed operator bordered  by adequate operators of multiplication  as is described below.

1. We first consider Dirichlet boundary condition case. The unperturbed operator $H_{0}= - \ud^{2} / \ud x^{2}$ is taken with the domain the Sobolev space $W_{p}^{2}(\reals_{+})$ consisting of all functions $u\in L_{p}(\reals_{+})$ such that $u$, $u'$ are absolutely continuous with $u''\in L_{p}(\reals_{+})$ and $u(0)=0$. The operator $H_{0}$ is closed and $\sigma (H_{0}) = [ 0, \infty)$. For any $\lambda \in \complex \setminus [0, \infty)$ the resolvent $R(\lambda; H_{0}) = (H_{0} - \lambda I)^{-1}$ is the integral operator
\begin{equation} \label{eq:resolvent}
R(\lambda; H_{0})v(x) = - \frac{1}{2 i \mu} \int_{0}^{\infty} e^{i \mu |x- y|} v(y) \ud y + \frac{1}{2 i \mu}\int_{0}^{\infty} e^{i \mu (x+y)} v(y)\ud y,
\end{equation}
where $\mu=\lambda^{1/2}$ is chosen so that $\impart \mu >0$.

We denote by $A, B$ the operators of multiplication by $a, b$, respectively, and evaluate the norm of the bordered resolvent $BR(\lambda; H_{0})A$. For we choose $\beta>0$ and $\gamma>0$ such that the evaluations 
\begin{equation} \label{eq:norm a}
\| au \|_{\beta}\leq \| a \|_{r} \| u \|_{p}, \quad \beta^{-1} = r^{-1} + p^{-1},
\end{equation}
and 
\begin{equation} \label{eq:norm b}
\| bv\|_{p} \leq \| b\|_{s}\| v \|_{\gamma}, \quad p^{-1}= s^{-1}+\gamma^{-1}
\end{equation}
hold true (those can be obtained by the use of H\"older inequality). In this way, the operator is bounded viewed as an operator acting from $L_{p}(\reals_{+})$ to $L_{\beta}(\reals_{+})$ and, respectively, $B$ as a bounded operator from $L_{\gamma}(\reals_{+})$ to $L_{p}(\reals_{+})$.

Next, we let
$$k(x, y; \lambda) = - \frac{1}{2 i \mu} ( e^{i \mu |x-y|} - e^{i \mu (x+y)}), \quad 0<x, y < \infty ,$$
for the kernel of the resolvent $R(\lambda; H_{0})$ and proceed as follows. First, we take an $\alpha$, $1\leq \alpha < \infty$, and observe that 
$$\sup_{0<x<\infty} \| k(x, \cdot ; \lambda)\|_{\alpha} \leq 1/|\mu| (\alpha \impart \mu)^{1/\alpha}.$$

 In fact, for any $x$, $0 < x < \infty$, we have

$$ \int_{0}^{\infty} | e^{i \mu | x - y|} |^{\alpha}  \ud y =   
\int_{0}^{x} e^{- \alpha (\impart \mu) (x - y)}  \ud y  + \int_{x}^{\infty} e^{- \alpha (\impart \mu) (- x + y)}  \ud y  =$$

$$ =   \frac{1}{\alpha \impart \mu}  \left( 2 -  e^{- \alpha (\impart \mu) x} \right),$$
and
$$ \int_{0}^{\infty} | e^{i \mu (x + y)} |^{\alpha} \ud y =   \int_{0}^{\infty} e^{- \alpha (\impart \mu) ( x + y)}  \ud y 
 =   \frac{1}{\alpha \impart \mu}   e^{- \alpha (\impart \mu) x},$$

hence

$$\| k (x, \cdot; \lambda) \|_{\alpha}  = 
 \left( \int_{0}^{\infty}    \left| \frac{1}{2 i \mu}  \left(  e^{i \mu | x - y |}  -  e^{i \mu (x + y)} \right)     \right|^{\alpha}  \ud y \right)^{1/\alpha} \leq$$

$$\leq   \frac{1}{2 | \mu |}   \left( \left( \int_{0}^{\infty}  |  e^{i \mu | x - y |} |^{\alpha} \ud y   \right)^{1/\alpha} +   \left( \int_{0}^{\infty}  |  e^{i \mu ( x + y )} |^{\alpha} \ud y   \right)^{1/\alpha}   \right)  =$$

$$= \frac{1}{2 | \mu |}  \left( \left( \frac{1}{2 \impart \mu}  (2 -  e^{- \alpha (\impart \mu) x} )   \right)^{1/\alpha}  +  \left( \frac{1}{2 \impart \mu} e^{- \alpha (\impart \mu) x}   \right)^{1/\alpha}     \right) =$$

$$= \frac{1}{2 | \mu | (\alpha \impart \mu)^{ 1/\alpha}}  \left( (2 -  e^{- \alpha (\impart \mu) x} )^{1/\alpha} + e^{- (\impart \mu) x}  \right),$$
i.e.,
$$\| k (x, \cdot; \lambda) \|_{\alpha} \leq \frac{1}{2 | \mu | (\alpha \impart \mu)^{ 1/\alpha}}  \left( (2 -  e^{- \alpha (\impart \mu) x} )^{1/\alpha} + e^{- (\impart \mu) x}  \right).$$ 
 The optimal value of the right member for varying $x$, $0 < x < \infty$, is equal to $1/| \mu | (\alpha \impart \mu)^{1/\alpha}$, and, thus, the desired inequality follows.

By Minkowski's inequality, it follows
$$\| R(\lambda ; H_{0}) v \|_{\alpha} = \left( \int_{0}^{\infty} \left| \int_{0}^{\infty} k(x,y;\lambda)v(y)\ud y \right|^{\alpha}\ud x \right)^{1/\alpha} \leq$$
$$\leq \int_{0}^{\infty} \left( \int_{0}^{\infty} |k(x,y; \lambda)|^{\alpha} \ud x \right)^{1/\alpha} |v(y)| \ud y \leq \sup_{0<y<\infty} \| k(\cdot, y; \lambda)\|_{\alpha} \| v\|_{1},$$
and, since the variables in the kernel $k(x, y; \lambda)$ are equal right, one has 
\begin{equation} \label{eq:resker}
\| R(\lambda ; H_{0}) v \|_{\alpha} \leq (1/ | \mu | (\alpha \impart \mu)^{1/ \alpha}) \| v\|_{1}.
\end{equation}
On the other hand, by H\"older's inequality, there holds
$$|R(\lambda; H_{0})v(x)| = \left| \int_{0}^{\infty} k(x, y; \lambda) v(y) \ud y \right| \leq$$
$$\left( \int_{0}^{\infty} | k(x, y ; \lambda)|^{\alpha} \ud y \right)^{1/\alpha} \left( \int_{0}^{\infty} |v(y)|^{\alpha'} \ud y \right)^{1/\alpha'} = \| k(x, \cdot; \lambda)\|_{\alpha} \| v\|_{\alpha'},$$
that yields that 
\begin{equation} \label{eq:resker1}
\| R(\lambda; H_{0}) v\|_{\infty} \leq (1/ |\mu| (\alpha \impart \mu)^{1/\alpha}) \| v \|_{\alpha'}.
\end{equation}

The evaluation \eqref{eq:resker} means that the resolvent operator $R(\lambda; H_{0})$ is bounded regarded as an operator from $L_{1}(\reals_{+})$ to $L_{\alpha}(\reals_{+})$ while \eqref{eq:resker1} means the boundedness of $R(\lambda; H_{0})$ as an operator from $L_{\alpha'}(\reals_{+})$ to $L_{\infty}(\reals_{+})$. In both cases its norm is bounded by $1/ |\mu| (\alpha \impart \mu)^{1/\alpha}$. By applying the Riesz-Thorin interpolation theorem (see, for instance, \cite{bergh-lofstrom}; Theorem 1.1.1) we conclude that the resolvent operator $R(\lambda; H_{0})$ is bounded from $L_{\beta}(\reals_{+})$ to $L_{\gamma}(\reals_{+})$ provided that 
$$\frac{1}{\beta} = \frac{1-\theta}{1} + \frac{\theta}{\alpha'}, \quad \frac{1}{\gamma} = \frac{1-\theta}{\alpha} + \frac{0}{\infty}, \quad 0<\theta <1.$$
Moreover, the corresponding value of its norm does not exceed $1/ |\mu | (\alpha \impart \mu)^{1/\alpha}$. Eliminating $\theta$, we find
$$\alpha^{-1} + \beta^{-1} = \gamma^{-1} + 1,$$
which, in view of restriction in \eqref{eq:norm a} and \eqref{eq:norm b}, implies $\alpha = (1-r^{-1} - s^{-1})^{-1}$. Note that due to the fact that $1\leq \alpha < \infty$ it must be $0\leq r^{-1}+s^{-1}<1$. In these conditions we obtain the following estimate
$$\| BR(\lambda; H_{0}) Au \|_{p} \leq (1/ |\mu| (\alpha \impart \mu)^{1/\alpha}) \| a \|_{r} \| b \|_{s} \| u \|_{p},$$
and, therefore, for any eigenvalue $\lambda \in \complex \setminus \reals_{+}$ of $H$ it should be fulfilled 
$$|\mu| (\alpha \impart \mu)^{1/\alpha} \leq \| a \|_{r} \| b \|_{s},$$
that, by letting $\lambda = |\lambda| e^{i \theta}$, $0<\theta < 2\pi$, provides to the following estimate
\begin{equation} \label{eq:lambda}
|\lambda |^{1+\alpha} \leq (\alpha \sin (\theta/ 2))^{-2} \| a \|_{r}^{2\alpha} \| b \|_{s}^{2\alpha}.
\end{equation}
We have proved the following result.
\begin{thm} \label{thm:result1}
Let $1<p< \infty$, $0<r \leq \infty$, $p\leq s \leq \infty$, $r^{-1}+s^{-1} < 1$, and assume $q=ab$, where $a \in L_{r}(\reals_{+})$ and $b \in L_{s}(\reals_{+})$. Then, for any eigenvalue $\lambda \in \complex \setminus \reals_{+}$ of the operator $H$, considered acting in $L_{p}(\reals_{+})$, there holds \eqref{eq:lambda}.
\end{thm}

Diverse estimates useful in applications can be derived from the above general result. If in Theorem \ref{thm:result1} is taken $r=s$, there obtains the following result.
\begin{cor} \label{cor:cor1}
Suppose $q=ab$, where $a, b \in L_{r}(\reals_{+})$ with $r>2$ if $1< p \leq 2$ and $p\leq r \leq \infty$ if $p>2$. Then for any eigenvalue $\lambda \in \complex \setminus \reals_{+}$ of $H$, considered acting in $L_{p}(\reals_{+})$, there holds
\begin{equation} \label{eq:cor1}
|\lambda|^{r-1} \leq \left( \frac{r}{r-2} \sin \frac{\theta}{2} \right)^{2-r} \| a \|_{r}^{r} \| b \|_{r}^{r}.
\end{equation}
\end{cor}
The following particular case presents peculiar interest in many situations.
\begin{cor} \label{cor:cor2}
Let $\gamma > 1/2$ if $1< p \leq 2$ and $2\gamma \geq p-1$ if $p>2$, and suppose 
$$q\in L_{\gamma + 1/2} (\reals_{+}).$$
Then any eigenvalue $\lambda \in \complex \setminus \reals_{+}$ of the operator $H$ in $L_{p}(\reals_{+})$ satisfies 
\begin{equation} \label{eq:cor2}
|\lambda|^{\gamma} \leq \left( \frac{2\gamma + 1}{2 \gamma -1} \sin \frac{\theta}{2}\right)^{1/2 - \gamma} \int_{0}^{\infty} |q(x)|^{\gamma +1/2} \ud x.
\end{equation}
\end{cor}
\begin{proof}
In Corollary \ref{cor:cor1} it suffices to let $r=2\gamma +1$ and take $a(x)=|q(x)|^{1/2}$ and $b(x)=(sgn q(x))|q(x)|^{1/2}$, where $sgn q(x) = q(x) / |q(x)|$ of $q(x)\neq 0$ and $sgn q(x) = 0$ if $q(x)=0$. 
\end{proof}
\begin{rem} \label{rem:rem1}
For the self-adjoint case can be occurred only negative eigenvalues of $H$, and thus \eqref{eq:cor2} becomes
\begin{equation} \label{eq:rem1}
|\lambda|^{\gamma} \leq \left( \frac{2\gamma -1}{2\gamma +1}\right)^{\gamma -1/2} \int_{0}^{\infty} |q(x)|^{\gamma+1/2} \ud x.
\end{equation}
\end{rem}
Similar estimates for whole-line operators were pointed out in \cite{keller} or \cite{lieb-thirring}. For related results and discussion in other contexts see also \cite{davies-nath}, \cite{frank-laptev-lieb-seiringer}, \cite{frank-laptev-seiringer} and \cite{laptev-safronov}.

Estimates involving decaying potentials can be also obtained directly from the general results. So, if we take in \eqref{eq:cor1} $a(x) = (1+x)^{-\tau}$ and $b(x)=(1+x)^{\tau}q(x)$, we obtain the following result.
\begin{cor} \label{cor:cor3}
Suppose $(1+x)^{\tau} q \in L_{r}(\reals_{+})$ with $\tau r >1$ and $r$ as in Corollary \ref{cor:cor1}. Then any eigenvalue $\lambda \in \complex \setminus \reals_{+}$ of the operator $H$, considered acting in $L_{p}(\reals_{+})$, satisfies
\begin{equation} \label{eq:cor3}
|\lambda|^{r-1} \leq \frac{1}{\tau r - 1} \left( \frac{r}{r-2} \sin \frac{\theta}{2} \right)^{2-r} \int_{0}^{\infty} |(1+x)^{\tau} q(x)|^{r} \ud x.
\end{equation}
\end{cor}

The following is also a simple consequence of above general results.
\begin{cor} \label{cor:cor4}
Let $r$ and $p$ be as in Corollary \ref{cor:cor1}, and suppose $e^{\tau x}q \in L_{r}(\reals_{+})$ for $\tau>0$. Then, any eigenvalue $\lambda \in \complex \setminus \reals_{+}$ of the operator $H$, considered acting in $L_{p}(\reals_{+})$ satisfies
\begin{equation} \label{eq:cor4}
|\lambda|^{r-1} \leq \frac{1}{\tau r} \left( \frac{r}{r-2} \sin \frac{\theta}{2} \right)^{2-r} \int_{0}^{\infty} e^{\tau r x} |q(x)|^{r} \ud x.
\end{equation}
\end{cor} 

For the extremal case $\alpha = \infty$ we have 
$$\sup_{0<x<\infty} \| k(x, \cdot; \lambda)\|_{\infty} = \sup_{0<x,y< \infty} \frac{1}{2 |\mu|} | e^{i \mu |x-y|} - e^{i \mu (x+y)}|.$$
This supremum has been computed in \cite{frank-laptev-seiringer} (see \cite{frank-laptev-seiringer}, proof of Theorem 1.1, and also Lemma 1.3). It turns out that
$$\sup_{0<x, y<\infty} | e^{i\mu |x-y|} - e^{i \mu (x+y)}| = g(\cot (\theta/2)), \quad \theta=\arg \lambda,$$
where
$$g(a):= \sup_{0<x, y<\infty} |e^{i a y} - e^{-y}|$$
($g$ is an even function: $g(-a)=g(a)$).
Let us show this fact for the sake of completeness. It follows from the following simple relations:
$$\sup_{0<x, y < \infty} |e^{i \mu |x-y|} - e^{i \mu (x+y)}| = \sup_{0<y<x} |e^{i \mu (x-y)} - e^{i \mu (x+y)}| =$$
$$= \sup_{y>0} |1- e^{2i\mu y}| = \sup_{y>0} |e^{-i(\cot(\theta/2))y} - e^{-y}| = g(\cot(\theta/2)).$$
Thus,
$$\sup \| k(x, \cdot; \lambda) \|_{\infty} = \frac{1}{2|\mu|}g(\cot(\theta/2)),$$
and, therefore, the resolvent operator $R(\lambda; H_{0})$ is bounded from $L_{1}(\reals_{+})$ to $L_{\infty}(\reals_{+})$, and 
\begin{equation}  \label{eq:norm res}
\| R(\lambda; H_{0})v \|_{\infty} \leq \frac{1}{2|\mu|} g(\cot (\theta/2)) \| v \|_{1}.
\end{equation}
In this case it should be taken $\beta = 1$, $\gamma = \infty$, then \eqref{eq:norm a} and \eqref{eq:norm b} held for $r=p^{'}$ and $s=p$ that together with \eqref{eq:norm res} implies
$$\| BR(\lambda; H_{0}) Au \|_{p} \leq \frac{1}{2|\mu|} g(\cot (\theta/2))\| a \|_{p'} \| b \|_{p} \| u \|_{p}.$$
Therefore, for any eigenvalue $\lambda \in \complex \setminus \reals_{+}$ of the operator $H_{0}$, it should be held (changed $a$ with $b$)
\begin{equation}  \label{eq:a with b}
|\lambda|^{1/2} \leq \frac{1}{2} g(\cot(\theta/2)) \| a \|_{p} \| b \|_{p'}.
\end{equation}

Thus, there holds the following result.
\begin{thm} \label{thm:result2}
Let $1<p<\infty$, and let $q=ab$ with $a\in L_{p}(\reals_{+})$ and $b\in L_{p'}(\reals_{+})$. Then any eigenvalue $\lambda \in \complex \setminus \reals_{+}$ of the operator $H$, considered acting in $L_{p}(\reals_{+})$, satisfies \eqref{eq:a with b}.
\end{thm}

In particular, for the Hilbert space case $p=2$ there holds
\begin{equation}  \label{eq:result2}
|\lambda|^{1/2} \leq \frac{1}{2} g(\cot (\theta/2)) \| a \|_{2} \| b \|_{2}.
\end{equation}
\begin{rem}  \label{rem:rem2}
If in \eqref{eq:result2} it is taken $a(x) = |q(x)|^{1/2}$ and $b(x) = (sgn q(x)) |q(x)|^{1/2}$, then 
$$\| a \|_{2}^{2} = \| b \|_{2}^{2} = \int_{0}^{\infty} |q(x)| \ud x,$$
and estimate \eqref{eq:result2} becomes
$$|\lambda|^{1/2} \leq \frac{1}{2} g(\cot (\theta/2)) \int_{0}^{\infty} |q(x)| \ud x.$$
\end{rem}
This result was established in \cite{frank-laptev-seiringer} (it is presented in Theorem 1.1 \cite{frank-laptev-seiringer} as the main result).
\begin{rem}  \label{rem:rem3}
From results presented in Theorem \ref{thm:result2} can be established various special estimates useful for applications. For instance, arguing as in the case of Corollary \ref{cor:cor3}, it can be derived the following estimate
$$|\lambda|^{1/2} \leq \frac{1}{2} g(\cot(\theta/2))(p'\tau - 1)^{-1/p'} \| (1+x)^{\tau}q \|_{p}$$
provided that $p'\tau > 1$ and $(1+x)^{\tau}q \in L_{p}(\reals_{+})$.
\end{rem}

2. By applying the same arguments there can be obtained related estimates for eigenvalues of the operator $H= -\ud^{2}/\ud x^{2} + q(x)$ considered with general boundary conditions like $u'(0) = \sigma u(0)$ ($0\leq \sigma < \infty$; in case $\sigma = \infty$ it is taken the Dirichlet condition; $\sigma = 0$ corresponds to the Neumann boundary condition $u'(0)=0$). We attach to this general situation all conventions made above for the Dirichlet boundary condition case concerning the exact definition of the perturbed operator $H$. In what follows, the corresponding perturbed operator is denoted by $H_{\sigma}$ (it will be no confusion with the notation $H_{0}$ used as unperturbed operator and the operator $H$ corresponding to the Neumann boundary condition case $\sigma = 0$). In order to apply the arguments used above, we first note that the resolvent operator of the unperturbed operator (with general boundary conditions) is an integral operator with the kernel
$$ k_{\sigma} (x, y; \lambda) = - \frac{1}{2\mu} \left( e^{i \mu |x-y|} - \frac{\sigma+i\mu}{\sigma- i\mu} e^{i \mu (x+y)} \right), \quad 0<x, y< \infty.$$

Next, we take $\alpha$, $1\leq \alpha < \infty$, and observe that
$$\sup_{0<x<\infty} \| k_{\sigma} (x, \cdot; \lambda) \|_{\alpha} \leq 1/|\mu| (\alpha \impart \mu)^{1/\alpha}.$$
We take $\impart \mu > 0$ and then $|(\sigma + i \mu) / (\sigma - i \mu)^{-1}|\leq 1$, and arguments similar to that used in proving the estimation for the Dirichlet boundary condition case (when $\sigma = \infty$) are applied. 

Thus, as is seen, for the operator $H_{\sigma}$, that is, for the case of general boundary conditions, the result given by Theorem \ref{thm:result1} remains valid as well. We formulate the corresponding result in a separate theorem.
\begin{thm} \label{thm:result3}
Under conditions of Theorem \ref{thm:result1} for any eigenvalue $\lambda \in \complex \setminus \reals_{+}$ of the operator $H_{\sigma}$, considered acting in $L_{p}(\reals_{+})$, an estimate like \eqref{eq:lambda} holds true. In particular, for negative eigenvalues $\lambda$ of $H_{\sigma}$, there holds
$$|\lambda|^{1+\alpha} \leq \alpha^{-2} \| a \|_{r}^{2\alpha} \| b \|_{s}^{2\alpha}.$$
\end{thm}

For the extremal case $\alpha= \infty$ we have
$$\sup_{0<x<\infty} \| k_{\sigma}(x, \cdot; \lambda) \|_{\infty} = \frac{1}{2|\mu|} g_{\sigma} (\cot (\theta/2)),$$
$g_{\sigma}$ instead of $g$, where
$$g_{\sigma} (a):= \sup_{0<x<\infty} \left| e^{iay} - \frac{\sigma+ i \mu}{\sigma -i\mu} e^{-y}\right|, \quad a\in \reals.$$

Accordingly, the following result holds true.
\begin{thm} \label{thm:result4}
Under the conditions of Theorem \ref{thm:result2} for any eigenvalue $\lambda \in \complex \setminus \reals_{+}$ of the operator $H_{\sigma}$, considered acting in $L_{p}(\reals_{+})$, there holds
$$|\lambda|^{1/2}\leq \frac{1}{2} g_{\sigma}(\cot(\theta/2)) \| a \|_{p} \| b \|_{p}.$$
\end{thm}
\begin{rem}  \label{rem:rem4}
The result given by Theorem \ref{thm:result4} for the Hilbert space case $p=2$ and $a(x)= |q(x)|^{1/2}$, $b(x)= (sgn q(x)) |q(x)|^{1/2}$ was mentioned in \cite{frank-laptev-seiringer} (see Proposition 1.5 \cite{frank-laptev-seiringer}).
\end{rem}

%%%%%%%%%%%%%%%%%%%%%%%%%%%%%%%%%%%%%%%%%%%%%%%%%%%%%%%%
\section{The case of potentials from weak Lebesgue's spaces}\label{sec:potentials from weak}
%%%%%%%%%%%%%%%%%%%%%%%%%%%%%%%%%%%%%%%%%%%%%%%%%%%%%%%%

Estimates for the perturbed eigenvalues can be obtained under slightly weakened conditions on the potentials. It turns out that it can be involved potentials belonging to weak Lebesgue's spaces. To be more precisely we consider a Schr\"odinger operator $H= - \ud^{2}/ \ud x^{2} + q(x)$, where the potential $q$ is written as a product $q=ab$ with $a\in L_{r, w}(\reals_{+})$ and $b \in L_{s, w}(\reals_{+})$ (we will use $L_{r, w}$ to denote the so-called weak $L_{r}$-spaces). The operator $H$ will be considered acting in the space $L_{p}(\reals_{+})$ ($1<p< \infty$) and subjected with the Dirichlet boundary condition (there will no loss of generality in supposing only the Dirichlet boundary condition).

We recall that the weak $L_{r}$-space $L_{r, w}(\reals_{+})$ ($0<r< \infty$) is the space consisting of all measurable functions on $\reals_{+}$ such that
$$\| f \|_{r, w} := \sup_{t>0} (t^{r} \lambda_{f}(t))^{1/r} < \infty,$$
where $\lambda_{f}$ denotes the distribution function of $f$, namely, $\lambda_{f}(t)= \mu (\{x\in \reals_{+}: |f(x)| > t\}),$ $0<t<\infty$ (here $\mu$ is the standard Lebesgue measure on $\reals_{+}$).
The spaces $L_{r, w}(\reals_{+})$ are special cases of the more general Lorentz spaces $L_{p, r} (\reals_{+})$ which will also be needed. $L_{p, r}(\reals_{+})$ $(0<p, r< \infty)$ is defined as the space of all measurable functions $f$ on $\reals_{+}$ for which
$$\| f \|_{p, r}^{r}:= \int_{0}^{\infty} t^{r}(\lambda_{f}(t))^{r/p} \frac{\ud t}{t} < \infty.$$
Note that $L_{r, r} (\reals_{+}) = L_{r}(\reals_{+})$, and it will be convenient to let $L_{\infty, r}(\reals_{+})= L_{\infty}(\reals_{+})$.

As in previous section we let $A, B$ denote the operators of multiplication by $a, b$, respectively. In view of $a\in L_{r, w}(\reals_{+})$ and $b\in L_{s, w}(\reals_{+})$, as was assumed, we can apply a result of O'Neil \cite{o'neil} due to of which there can be chosen $\beta>0$ and $\gamma>0$ such that the multiplication operator $A$ to be bounded from $L_{p, p}(\reals_{+}) (=L_{p}(\reals_{+}))$ to $L_{\beta, p}(\reals_{+})$ and, respectively, $B$ to be bounded from $L_{\gamma, p}(\reals_{+})$ to $L_{p, p}(\reals_{+})$ and, moreover, 
\begin{equation} \label{eq:norm A}
\| Au \|_{\beta, p} \leq c \| a \|_{r, w} \| u \|_{p}, \quad \beta^{-1} = r^{-1} + p^{-1},
\end{equation}
and
\begin{equation}  \label{eq:norm B}
\| Bv \|_{p} \leq c \| b \|_{s, w} \| v \|_{\gamma, p}, \quad p^{-1} = s^{-1}+\gamma^{-1}.
\end{equation}
Note that in \eqref{eq:norm A} and \eqref{eq:norm B} the constants in general are distinct, but depending only on $r, p$ and $s, p$, respectively.

Next, as was shown, the resolvent operator $R(\lambda; H_{0})$ of unperturbed operator $H_{0}$ acts as a bounded operator from $L_{1}(\reals_{+})$ to $L_{\alpha}(\reals_{+})$ and, simultaneously, from $L_{\alpha'}(\reals_{+})$ to $L_{1}(\reals_{+})$ for any $\alpha$, $1\leq \alpha <\infty$. Besides, in both cases the bound for the norm of $R(\lambda; H_{0})$ does not exceed $1/ |\mu| (\alpha \impart \mu)^{1/ \alpha}$. By applying the interpolation functor $K_{\theta, p}$ with $0<\theta < 1$ (cf., \cite{bergh-lofstrom} or \cite{triebel}), we obtain that $R(\lambda; H_{0})$ acts also as a bounded operator from $L_{\beta, p}(\reals_{+})$ into $L_{\gamma, p}(\reals_{+})$ provided that 
$$\frac{1}{\beta} = \frac{1-\theta}{1} + \frac{\theta}{\alpha'}, \quad \frac{1}{\gamma}=\frac{1-\theta}{\alpha}+ \frac{\theta}{\infty}.$$
Moreover,
\begin{equation}  \label{eq:resolv norm}
\| R(\lambda; H_{0})v \|_{\gamma, p} \leq \left( C/ |\mu| (\alpha \impart \mu)^{1/\alpha}\right) \| v \|_{\beta, p},
\end{equation}
where $C$ is a positive constant depending only on $p, \gamma$ and $\beta$ occurred after interpolation. In view of \eqref{eq:norm A}, \eqref{eq:norm B} and \eqref{eq:resolv norm} we conclude that, under our assumption, there holds
$$\| BR(\lambda; H_{0})Au \|_{p} \leq \left( C/ |\mu| (\alpha \impart \mu)^{1/\alpha} \right) \| a \|_{r, w} \| b \|_{s, w} \| u \|_{p}$$
with a positive constant $C$ depending only on $p, r$ and $s$.

In this manner, we obtain an estimate like \eqref{eq:lambda}, namely
\begin{equation} \label{eq:lambda sin}
|\lambda|^{1+\alpha} \leq C (\alpha \sin(\theta/2))^{-2} \| a \|_{r, w}^{2\alpha} \| b \|_{s, w}^{2\alpha},
\end{equation}
under more weaker conditions on the potential $q$ than those required in Theorem \ref{thm:result1}. In \eqref{eq:lambda sin}, as in \eqref{eq:lambda}, $\alpha= (1-r^{-1} - s^{-1})^{-1}$ with the same restrictions on $r$ and $s$. We have proved the following result.
\begin{thm}  \label{thm:result5}
Let $p, r, s$ be as in Theorem \ref{thm:result1}, and suppose $q=ab$, where $a\in L_{r, w} (\reals_{+})$ and $b\in L_{s, w}(\reals_{+})$. Then, any eigenvalues $\lambda \in \complex \setminus \reals_{+}$ of the operator $H$, considered acting in $L_{p}(\reals_{+})$, satisfies \eqref{eq:lambda sin}.
\end{thm} 

Similarly as in previous section diverse estimates for eigenvalues useful in applications can be derived from general result given by Theorem \ref{thm:result5}. Here we restrict ourselves to remark only  a version accommodated for potentials from weak Lebesgue's classes of the result presented in Corollary \ref{cor:cor2}.
\begin{cor}  \label{cor:cor5}
Let $2\gamma > 1$ if $1<p \leq 2$ and $2\gamma > p-1$ if $p>2$, and suppose $q\in L_{\gamma+ 1/2, w}(\reals_{+})$. Then any eigenvalue $\lambda \in \complex \setminus \reals_{+}$ of the operator $H$, considered acting in $L_{p}(\reals_{+})$, satisfies
\begin{equation}  \label{eq:cor5}
|\lambda|^{\gamma} \leq C \sup_{t>0} \left(  t^{\gamma + 1/2} \lambda_{q} (t) \right), 
\end{equation}
where $C= C(p, \gamma, \theta)$ is a positive constant depending only on $p, \gamma$ and $\theta$ ($\theta = \arg \lambda$). 
\end{cor} 
\begin{rem}  \label{rem:rem5}
In \eqref{eq:cor5} the value of $C$ can be controlled.
\end{rem}

%%%%%%
%%%
%%%   Acknowledgments
%%%
%%%%%%

\section*{Acknowledgments}

The author wishes to express her gratitudes to Professor Ari Laptev for formulating the problem and for many useful discussions.

%%%%%%
%%%
%%%   Bibliography
%%%
%%%%%%
\bibliography{estimates-on-half-line}
\bibliographystyle{alpha}

\end{document}